\documentclass{article}
\usepackage[T1]{fontenc}
\usepackage[utf8]{inputenc}
\usepackage{geometry}
\usepackage{amsmath}
\usepackage{amsthm}
\usepackage[shortlabels]{enumitem}
\usepackage[dvipsnames]{xcolor} 
\usepackage[color=Orange!50!white, textwidth=22mm]{todonotes}
\usepackage[unicode=true, bookmarks=false, breaklinks=false,
 pdfborder={0 0 1},
 colorlinks=false,
 hidelinks=true]
 {hyperref}
\usepackage[noabbrev,capitalize,nameinlink]{cleveref}
\usepackage{comment}
\makeatletter
\theoremstyle{plain}
\newtheorem{thm}{\protect\theoremname}
\theoremstyle{plain}
\newtheorem{lem}[thm]{\protect\lemmaname}

\theoremstyle{remark}

\theoremstyle{plain}

\theoremstyle{plain}

\theoremstyle{plain}

\theoremstyle{plain}

\usepackage{amssymb}

\makeatother

\providecommand{\claimname}{Claim}
\providecommand{\corollaryname}{Corollary}
\providecommand{\lemmaname}{Lemma}
\providecommand{\theoremname}{Theorem}
\providecommand{\problemname}{Problem}
\providecommand{\definitionname}{Definition}

\renewcommand{\epsilon}{\varepsilon}

\newcommand{\mH}{\mathcal{H}}
\newcommand{\mL}{\mathcal{L}}
\newcommand{\mR}{\mathcal{R}}
\newcommand{\mD}{\mathcal{D}}
\newcommand{\mI}{\mathcal{I}}
\newcommand{\mC}{\mathcal{C}}
\newcommand{\mB}{\mathcal{B}}

\newcommand{\maF}{\mathbb{F}}

\title{Multipartite nearly orthogonal sets over finite fields}
\author{Rajko Nenadov \footnote{School of Computer Science, University of Auckland, New Zealand. Email: rajko.nenadov@auckland.ac.nz. Research supported by the New Zealand Marsden Fund.} \and Lander Verlinde \footnote{School of Computer Science, University of Auckland, New Zealand. Email: lver263@aucklanduni.ac.nz.}}
\date{}

\begin{document}

\maketitle

\begin{abstract}       
    For a field $\maF$ and integers $d, k$ and $\ell$, a set $A \subseteq \maF^d$ is called $(k,\ell)$-nearly orthogonal if all vectors in $A$ are non-self-orthogonal and every $k+1$ vectors in $A$ contain $\ell + 1$ pairwise orthogonal vectors. 
    Recently, Haviv, Mattheus, Milojevi\'{c} and Wigderson have improved the lower bound on nearly orthogonal sets over finite fields, using counting arguments and a hypergraph container lemma. They showed that for every prime $p$ and an integer $\ell$, there is a constant $\delta(p,\ell)$ such that for every field $\maF$ of characteristic $p$ and for all integers $d \geq k \geq \ell + 1$, $\maF^d$ contains a $(k,\ell)$-nearly orthogonal set of size $d^{\delta k / \log k}$. This nearly matches an upper bound $\binom{d+k}{k}$ coming from Ramsey theory. Moreover, they proved the same lower bound for the size of a largest set $A$ where for any two subsets of $A$ of size $k+1$ each, there is a vector in one of the subsets orthogonal to a vector in the other one. We prove a common generalisation of this result, showing that essentially the same lower bound holds for the size of a largest set $A \subseteq \mathbb{F}^d$ with the stronger property that given any family of subsets $A_1, \ldots, A_{\ell+1} \subseteq A$, each of size $k+1$, we can find a vector in each $A_i$ such that they are all pairwise orthogonal. Rather than combining both counting and container arguments, we make use of a multipartite asymmetric container lemma that allows for non-uniform co-degree conditions. This lemma was first discovered by Campos, Coulson, Serra and Wötzel, and we provide a new and short proof for this lemma.\\
    

\end{abstract}

\section{Introduction}
Consider an arbitrary field $\maF$ and integers $d, k,\ell$ such that $k \geq \ell$. A set $A \subseteq \maF^d$ of non-self-orthogonal vectors such that for any $k+1$ vectors there is a subset of size $\ell+1$ of pairwise orthogonal vectors is called a $(k,\ell)$-nearly orthogonal set. The largest size of such a set is denoted by $\alpha(d,k,\ell, \maF)$. The question of determining this number for $\maF = \mathbb{R}$ was posed by Erd\H{o}s in 1988 (see \cite{chawin2024nearly}). 

An easy lower bound when $\ell =1$ can be found by considering the vectors of $k$ pairwise disjoint orthogonal bases of $\mathbb{R}^d$, showing that the minimal size is at least $k d$. Erd\H{o}s conjectured this to be tight for $k = 2$, which was proven in 1991 by Rosenfeld \cite{rosenfeldalmost}. However, this lower bound is not tight for larger $k$, as was shown by Füredi and Stanley \cite{furedi1992sets}, proving the existence of a (5,1)-nearly-orthogonal set in $\mathbb{R}^4$ of size 24. A more general lower bound was given by Alon and Szegedy \cite{alon1999large} (extending a result of Frankl and R\"odl \cite{frankl87forbidden}), showing that for every integer $\ell$, there exists a constant $\delta(\ell) > 0$ such that whenever $k \geq \max\{\ell, 3\}$, $\alpha(d,k,\ell, \mathbb{R}) \geq d^{\delta \log k/ \log \log k}$. 

More recently, Balla~\cite{balla2023orthonormal} considered a bipartite version of this problem when $\ell = 1$.
For integers $d,k$ and a field $\maF$, let $\beta(d,k,\maF)$ denote the maximum size of a set $A \subseteq \maF^d$ of non-self-orthogonal vectors such that for any two subsets $A_1, A_2 \subseteq A$ of size $k+1$ each, there exist $v_1 \in A_1$ and $v_2 \in A_2$ that are orthogonal. It clearly follows that $\alpha(d,k,1,\maF) \geq \beta(d,k,\maF)$, as one can consider $A_1 = A_2$. For this extended notion Balla proved that for all integers $d$ and $k \geq 3$, there exists a constant $\delta > 0$ such that $\beta(d,k,\mathbb{R}) \geq d^{\delta \log k / \log \log k}$, matching the lower bound on $\alpha(d,k,1,\mathbb{R)}$ by Alon and Szegedy. Nevertheless, this lower bound is still rather far away from the best current upper bound $\alpha(d,k,1,\mathbb{R}) = O(d^{(k+1)/3)})$ for fixed $k$, proven by Balla, Letzter and Sudakov \cite{balla2020H-free}.  

Besides $\mathbb{R}$, the maximal size of nearly orthogonal sets has also been studied in finite fields. This was first motivated by Codenotti, Pudlák and Resta \cite{codenotti2000some}, who showed the link between nearly-orthogonal sets and computational complexity of arithmetic circuits. 
Golovnev and Haviv \cite{golovnev2022generalized} proved the existence of a constant $\delta > 0$ such that, for infinitely many integers $d$, $\alpha(d, 2,1, \maF_2) \geq d^{1+\delta}$, showing a clear difference between nearly orthogonal sets in finite fields compared to the reals, where the upper bound by Balla et al. implied $\alpha(d,2,1,\mathbb R) = O(d)$. This result was improved by Bamberg, Bishnoi, Ihringer and Ravi~\cite{bamberg2024ramsey}, giving an explicit construction of a (2,1)-nearly orthogonal set in $\maF_2^d$ of size $d^{1.2895}$. Chawin and Haviv \cite{chawin2024nearly} considered the bipartite version and proved that for every prime $p$, there exists a positive constant $\delta(p)$ such that for every field $\maF$ of characteristic $p$ and for all integers $k \geq 2$ and $d \geq k^{1/(p-1)}$, it holds that $\beta(d,k,\mathbb{F}) \geq d^{\delta k^{1/(p-1)}/ \log k}$. Hence, for $p=2$, it follows that 
\begin{equation} \label{eq:F2}
    \alpha(d,k,1,\maF_2) \geq \beta(d,k,\maF_2) \geq d^{\Omega(k / \log k)}. 
\end{equation}
This almost matches an upper bound that follows easily from Ramsey theory, up to the logarithmic term in the exponent: consider an arbitrary $(k,1)$-almost orthogonal set in $\maF^d$ and construct its orthogonality graph, where the vertices correspond to the vectors and  there is an edge if the vectors are orthogonal. Clearly this graph does not contain a clique of size $d+1$, nor does it have an independent set of size $k+1$, because that would contradict that the set is $(k,1)$-almost orthogonal. Hence, the size of the set is upper bounded by the Ramsey number $R(d+1, k+1)$. The upper bound on $R(d+1, k+1)$ by Erd\H{o}s and Szekeres \cite{erdos1935combinatorial} gives $\alpha(d,k,1,\maF_2) < \binom{d+k}{k}$. For $d \gg k$, this implies $\alpha(d,k,1,\maF_2) \leq d^{\Theta(k)}$. 

In a recent paper, Haviv, Mattheus, Milojevi\'{c} and Wigderson \cite{haviv2024larger} obtained the bound matching \eqref{eq:F2} for fields of characteristic larger than 2, and also extended it to the non-bipartite setup for arbitrary $\ell > 1$. Specifically, they showed that for every prime $p$ and integer $\ell \geq 1$, there exist constants $\delta_1(p), \delta_2(p,\ell) > 0$, such that for every field $\maF$ of characteristic $p$ and for all integers $d \geq k \geq \ell +1$, it holds that     
    \begin{align*}
        \beta(d,k,\maF) &\geq d^{\delta_1 k/ \log k} \\
        \alpha(d,k,\ell,\maF) &\geq d^{\delta_2 k / \log k},
    \end{align*} 
    proving that $\alpha(d,k,\ell,\maF) \geq \beta(d,k,\maF)$ also holds.
We propose a generalisation of this result. Instead of considering $\alpha(d,k,\ell, \maF)$ and $\beta(d,k,\maF)$ separately, we consider a common generalisation $\beta(d,k,\ell, \maF)$ defined as the maximum size of a set $A \subseteq \maF$ of non-self-orthogonal vectors such that for any $\ell +1$ subsets $A_1, \dots, A_{\ell +1} \subseteq A$ of size $k+1$ each, there exist $v_1 \in A_1, \dots, v_{\ell+1} \in A_{\ell+1}$ that are all pairwise orthogonal. We prove analogous bounds on $\beta(d, k, \ell, \mathbb{F})$.

\begin{thm}
\label{thm: orthogonal result}
    For every prime $p$ and integer $\ell \geq 1$, there exists a constant $\delta = \delta(p,\ell) > 0$ such that for every field $\maF$ of characteristic $p$ and for all integers $d \geq k \geq \ell + 1$, it holds that 
    \[ \beta(d,k,\ell, \maF) \geq d^{\delta  k/ \log k}.  \]
\end{thm}
\noindent
Note that Theorem \ref{thm: orthogonal result} implies the result by Haviv et al. By taking $\ell =1$, we retrieve the bound on $\beta(d,k,\maF)$ and by setting $A_1 = \dots = A_{\ell +1}$, we retrieve the bound on $\alpha(d,k,\ell,\maF)$.

The main distinction between the proofs of the results is that we unify their methods to find the nearly-orthogonal sets and their bipartite variants. To be more precise, their proof relies on the fact that the orthogonality graph $G$ arising from finite fields has strong pseudorandom properties, and uses two different counting results: 
\begin{enumerate}
    \item  Using an argument based on ideas of Alon and R\"odl \cite{alon2005sharp}, Haviv et al. derived an upper bound on the number of pairs of sets in pseudorandom graphs with no edge across (used to bound $\beta$). 
    \item  Using hypergraph containers \cite{balogh2015independent,saxton2016online}, they upper bounded the number of sets of size $k$ without a copy of $K_\ell$ (used to bound $\alpha$). 
\end{enumerate}
Having these counting results at hand, the proof proceeds by showing that a random induced subgraph of $G$ avoids sets of certain size which do not contain a copy of $K_\ell$, and it avoids pairs of subsets without an edge across. To make this actually work, a tensor-product trick along the lines of Alon and Szegedy \cite{alon1999large} is applied.

We consolidate these two counting results by developing a container-type result for $\ell$-tuples of subsets without any clique $K_\ell$ that contains a vertex in each subset of the tuple (Lemma \ref{lemma:underpin}). We will call such a tuple a `bad' $\ell$-tuple, as this is exactly the structure that we want to avoid. In particular, rather than counting the number of $\ell$-tuples of $k$-subsets without such a copy of $K_\ell$, we construct a small family $\mC$ of small subsets with the property that for any bad $\ell$-tuple $(U_1, \ldots, U_\ell)$, there exists $S \in \mC$ such that $U_i \subseteq S$ for some $i \in [\ell]$. 

The proof of Lemma \ref{lemma:underpin} uses an asymmetric, multipartite hypergraph container lemma. This lemma was first developed by Campos, Coulson, Serra and Wötzel \cite[Theorem 2.1]{campos2023typical}. They came up with this lemma in the study of sets with bounded doubling factor. An iterated application of this lemma on the right hypergrapgh yields Lemma \ref{lemma:underpin}, as we will show in Section \ref{section: orthogonal sets}. In Section \ref{section: containers}, we first review this asymmetric container lemma. We provide a different formulation to the lemma, which is practically equivalent to the one by Campos et al. The advantage of this formulation is that allows us to give a considerably shorter proof of the lemma, hopefully further improving accessibility of asymmetric, multipartite containers. 

\section{Containers with varying co-degree conditions}
\label{section: containers}
Hypergraph containers, developed by Balogh, Morris and Samotij \cite{balogh2015independent} and Saxton and Thomason \cite{saxton2016online}, provide a framework for studying independent sets in hypergraphs. On a high level, the main idea of hypergraph containers is to show that given any independent set $I$ in a hypergraph $\mH$, one can carefully choose a small subset $F \subseteq I$, called a \emph{fingerprint}, such that just based on $F$ one can confine $I$ to some $C \subseteq V(\mH)$, called a \emph{container}. Importantly, the size of $C$ is bounded away from $|V(\mH)|$. This innocent looking statement has far reaching consequences. For a thorough introduction, see \cite{balogh18survey}. 

To make use of the standard hypergraph container framework, one typically needs that the given hypergraph has fairly uniform degree and co-degree distribution. This will not be the case in our application. In particular, our hypergraph is $\ell$-partite and degrees depend on which part the vertex belongs to. More generally, the number of edges sitting on a given set of vertices depends on which parts these vertices lie in. Therefore, we require an asymmetric multipartite container lemma. The asymmetric lemma of Campos, Coulson, Serra and Wötzel \cite{campos2023typical} is developed for exactly this case. In this section, we present and prove a slight reformulation of the lemma, which allows for a more concise and straightforward analysis of the container algorithm. The statement is inspired by the one used by Nenadov and Pham \cite{nenadov2024short}. 


Let $V$ be an arbitrary set. For a subset $T \subseteq V$, we let $\langle T \rangle$ denote the upset of $T$, i.e. $\langle T \rangle = \{ S \subseteq V: T \subseteq S\}$.

\begin{lem} \label{lemma:container}
    Let $\ell \in \mathbb{N}$ and let $\mH$ be an $\ell$-partite $\ell$-uniform hypergraph on the vertex set $(V_1, \dots, V_\ell)$. Suppose there exist $p_1, \dots, p_\ell \in (0,1]$, $K > 1$ and a probability measure $\nu$ over $2^{V}$, $V = \bigcup_i V_i$, supported on $E(\mH)$ such that for all sets $S \subseteq [\ell]$, $T \in \prod_{i \in S} V_i$ and $i \in S$, we have 
    \begin{equation}
    \label{eq:spreadness condition}
        \nu ( \langle T \rangle ) \leq \frac{K}{|V_i|} \prod_{j \in S \setminus \{ i \}} p_j.
    \end{equation} 
    Let $U = (U_1, \dots, U_\ell)$, $U_i \subseteq V_i$, be an arbitrary independent set in $\mathcal{H}$ such that $|U_j| \geq p_j |V_j|$ for all $j \in [\ell]$.
    Then, 
    \begin{enumerate}[(i)] 
        \item there is a unique $\ell$-tuple $\mathbf{F} = (\emptyset, F_2, \ldots, F_\ell)$ associated with $(U_1, \ldots, U_\ell)$, where $F_j \subseteq U_j$ and $|F_j| = |V_j|p_j$ for $2 \le j \le \ell$; \
        
        \item there is an index $i = i(\mathbf{F}) \in [\ell]$, a corresponding set $C_i = C_i(\mathbf{F}) \subseteq V_i$ and a constant $\zeta = \zeta(K,\ell) >0$ such that $|C_i| \leq (1-\zeta) |V_i|$ and $U_i \subseteq F_i \cup C_i$.
        
        \item \label{container:reconstruct} Furthermore, given any $U_1', \ldots, U_\ell'$ such that $F_i \subseteq U_i' \subseteq U_i$, the $\ell$-tuple $\mathbf{F} = (\emptyset, F_2, \ldots, F_\ell)$ associated with $(U_1, \ldots, U_\ell)$ is also associated with $(U_1', \ldots, U_\ell')$.
    \end{enumerate}
\end{lem}
\noindent
Whenever $\nu$ is the uniform probability measure over $\mH$, i.e. $\nu(X) = 1/e(\mH)$ if $X \in E(\mH)$ and $\nu(X) = 0$~otherwise, this condition is just a co-degree condition analogous to the one in \cite{balogh2015independent,saxton2016online}. Indeed, we have 
\[ \nu( \langle T \rangle) = \sum_{S \subseteq V: T \subseteq S} \nu(S) = \Delta_T \cdot \frac{1}{e(\mH)},  \]
where $\Delta_T $ denotes the codegree of the set $T$, that is, the number of hyperedges in $\mathcal{H}$ which contain $T$.
Moreover, it is exactly when considering the uniform measure, one retrieves the asymmetric container lemma of Campos et al. 
Our proof follows the recent proof of the original hypergraph containers by Nenadov and Pham closely.
\begin{proof}[Proof of Lemma \ref{lemma:container}]
    We prove the lemma by induction on $\ell$. To that end, first consider $\ell = 1$. Let $C_1 \subseteq V_1$ to be the set of all vertices $v \in V_1$ with $\nu(\{v\})  = 0$. Since $\nu$ is supported on $E(\mH)$, any independent set avoids $V_1 \setminus C_1$. There are at least $|V_1| / K$ vertices with non-zero measure, so the lemma holds for $\zeta = 1/K$.
    We proceed to the case where $\ell \geq 2$. We use the following notation: given a subset $X \subseteq V$, let $\mI_\mH(X)$ denote the set of all indices for which $V_i$ intersects $X$, that is, 
    \[ \mI_\mH(X) := \{i \in [\ell]: X \cap V_i \neq \emptyset \}.\]
    
    \paragraph{Constructing fingerprint $F_\ell$.} Set $F_\ell = \emptyset \subseteq U_\ell$, $\mL = \emptyset \subseteq 2^{V \setminus V_\ell}$, $\mD, \mH ' = \emptyset \subseteq \mH$ and $\mR = \mH$. Repeat the following steps $|V_\ell| p_\ell$ times: 
    \begin{itemize}
        \item Pick the vertex $v \in U_\ell \setminus F_\ell$ with the largest value $\nu(\langle v \rangle \cap \mR)$. Tie-breaking is done in some canonical way, that is, following an ordering of the vertices agreed upon at the beginning of the algorithm. Note that we can always pick such a vertex $v$ since $|U_\ell| \geq p_\ell |V_\ell|$. 
        
        \item Add $v$ to $F_\ell$ and set $\mH' = \mH' \cup (\langle v \rangle \cap \mR)$.

        \item For each $X \in 2^{V \setminus V_\ell} \setminus \mL$ of size $|X| \leq \ell -1$, such that there exists $s \in \mI_\mH(X)$ for which the following inequality holds:
        \begin{equation}
        \label{eq:deletion condition}
        \nu( \langle X \rangle \cap \mH') > \frac{K}{|V_s|} \cdot p_\ell \prod_{ i \in \mI_\mH(X) \setminus \{s\}} p_i,
        \end{equation}
        add $X$ to $\mL$, set $\mD = \mD \cup ( \langle X \rangle \cap \mR)$ and set $\mR = \mH \left[ V \setminus F_\ell \right] \setminus \mD$.
    \end{itemize}  

    \paragraph{Observations on the process.} Before we show how to choose other fingerprints $F_2, \ldots, F_{\ell-1}$ and construct a container $C_i$ for some $U_i$, we first make some observations about the described process. First of all, for any subset $X \subseteq V \setminus V_\ell$, because of \eqref{eq:spreadness condition}, the value of $\nu(\langle X \rangle \cap \mH')$ increases by at most 
    $$
        \nu(\langle X \cup v \rangle) \leq \frac{K}{|V_\ell|} \cdot p_\ell \prod_{i \in \mI_\mH(X)} p_i
    $$
    after adding $v$ to $F_\ell$. Note that the value of $\nu(\langle X \rangle \cap \mH')$ does not change anymore once $X$ is added to $\mL$, and this happens when \eqref{eq:deletion condition} is satisfied. Therefore, at the end of the process it holds that for every $X \subseteq V \setminus V_\ell$ of size $|X| \leq \ell -1$, including those that are in $\mL$, and for all $s \in \mI_\mH(X)$, we have

    \begin{equation}
    \label{eq:end of process}
    \nu ( \langle X \rangle \cap \mH') \leq \frac{2K}{|V_s|} \cdot p_\ell \prod_{i \in \mI_\mH(X) \setminus \{s\} } p_i.
    \end{equation}
    Secondly, given $U_\ell' \subseteq U_\ell$ such that $F_\ell \subseteq U_\ell$, if we start the process with $U'_\ell$ instead of $U_\ell$, it produces the same output $F_\ell$, as well as the same $\mH'$ and $\mR$.

    Next, we derive a lower bound on $\nu(\mH')$. To start, we know that for each edge $e \in \mD$, there exists $X \in \mL$ such that $e \in \langle X \rangle$. Thus, 
    \[ \sum_{X \in \mL} \nu(\langle X \rangle) \geq \nu( \mD). \]
    Also, we know by \eqref{eq:deletion condition}, that for each $X \in \mL$ there is an $s(X) \in \mI_\mH(X)$ such that

    $$
        \sum_{X \in \mL} \nu( \langle X \rangle \cap \mH') > \sum_{X \in \mL} \frac{K}{|V_{s(X)}|} p_\ell \prod_{i \in \mI_\mH(X) \setminus \{s(X)\}} p_i \\
        \stackrel{\eqref{eq:spreadness condition}}{\geq} p_\ell \sum_{X \in \mL} \nu( \langle X \rangle) \geq p_\ell \nu(\mD).
    $$
    Moreover, each edge $e \in \mH'$ contributes at most to $2^\ell$ terms in $\nu( \langle X \rangle \cap \mH')$. So, 

    \[ \nu( \mH') \geq 2^{-\ell} \sum_{X \in \mL} \nu( \langle X \rangle \cap \mH') > 2^{-\ell} p_\ell \nu(\mD).\]
    In conclusion, we have now derived a relationship between $\nu(\mH')$ and $\nu(\mD)$:
    
    \begin{equation}
    \label{eq:measure of deletion}
    \nu( \mH') > 2^{-\ell} p_\ell \nu(\mD).
    \end{equation}
    Secondly, we relate the measure on $\mH'$ to the measure on the vertices in $U_\ell$ that are not contained in the fingerprint $F_\ell$. Let $\mR_i$ denote the hypergraph $\mR$ at the moment when the $i$-th vertex $v_i$ is added to $F_\ell$. Since $\mR$ is non-increasing and we always chose $v_i$ with maximal value $\nu(\langle v \rangle \cap \mR)$, we have

    \[ \nu (\mH') = \sum_{i =1}^{|F_\ell|} \nu(\langle v_i \rangle \cap \mR_i) \geq \sum_{i = 1}^{|F_\ell|} \max_{v \in U_\ell \setminus F_\ell} \nu (\langle v \rangle \cap \mR). \] 
    Therefore, we can conclude that 

    \begin{equation}
    \label{eq:max degree condition}
    \nu( \mH') \geq |V_\ell| p_\ell \max_{v \in U_\ell \setminus F_\ell} \nu( \langle v \rangle \cap \mR).
    \end{equation}

    \paragraph{Constructing the container.} Set $\alpha = 2^{-\ell -2}$ and consider two cases. 
    \vspace{0.2cm}

\noindent
\textbf{Case 1:} $\nu(\mH') < \alpha p_\ell$. By \eqref{eq:measure of deletion}, we have $\nu(\mD) < \frac{1}{4}$, so, using $\mH = \mH' \cup \mR \cup \mD$, $\nu(\mR) > \frac{1}{2}$. By \eqref{eq:max degree condition} it holds that for every $v \in U_\ell \setminus F_\ell$, 

\[ \nu(\langle v \rangle \cap \mR) \leq \frac{\alpha}{|V_\ell|}. \]
We use this measure to define our container. Let $C_\ell$ be the set of all vertices $v \in V_\ell \setminus F_\ell$ such that $\nu( \langle v \rangle \cap \mR) \leq \frac{\alpha}{|V_\ell|}$. Clearly, we have $U_\ell \setminus F_\ell \subseteq C_\ell$. Moreover, by \eqref{eq:spreadness condition},

\[ \nu(\mR) \leq \sum_{v \in C_\ell} \nu(\langle v \rangle \cap \mR) + \sum_{w \in V_\ell \setminus (F_\ell \cup C_\ell)} \nu(\langle w \rangle \cap \mR) < \alpha + (|V_\ell| - |C_\ell|) \frac{K}{|V_\ell|}. \]
Hence, 
$$
    |C_\ell| < |V_\ell| - (\nu(\mR) - \alpha) \frac{|V_\ell|}{K} < (1- \zeta) |V_\ell|
$$ 
for $\zeta = 1/(4K)$. Take $F_j \subseteq U_j$ to be a subset of size $|V_j|p_j$ chosen in some canonical way, for $2 \le j < \ell$. These fingerprints are defined for the convenience of having all fingerprints of the specified size. So we have obtained the $\ell$-tuple of fingerprints $\mathbf{F} = (\emptyset, F_2, \dots, F_\ell)$ and index $i(\mathbf{F}) = \ell$ such that all conditions of the lemma are satisfied.

\vspace{0.2cm}

    \noindent
    \textbf{Case 2:} $\nu(\mH') \geq \alpha p_\ell$. We use the induction hypothesis to find fingerprints for $U_1, \dots, U_{\ell-1}$ and a container for one of these sets. Let $\mH''$ be the $(\ell - 1)$-uniform $(\ell - 1)$-partite hypergraph consisting of sets $X$ such that $X = H' \setminus F_\ell$ for some $H' \in \mH'$. Let $\nu'$ be the probability measure over $2^{V \setminus V_\ell}$ given by

    \[\nu'(X) \propto \left\{
        \begin{array}{ll}
        \nu((X \cup 2^{F_\ell}) \cap \mH') & \mbox{if } X \in \mH'' \\
        0 & \mbox{otherwise.}
        \end{array}
    \right. \]
    Consider some $X \in \mathcal{H}''$. For $s \in \mI_{\mH''}(X)$, we have 
$$
    \nu'(\langle X \rangle) = \frac{\nu(\langle X \rangle \cap \mH')}{\nu(\mH')}
    \stackrel{\eqref{eq:end of process}}{\leq} \left( \frac{2K}{|V_s|}p_\ell \prod_{i \in \mI_\mH(X) \setminus \{s\}}p_i \right) \frac{1}{\alpha p_\ell}
    \leq \frac{K'}{|V_s|} \prod_{i \in \mI_\mH(X) \setminus \{s\}} p_i,
$$
for $K' = 2K/\alpha$.
Furthermore, $(U_1, \dots, U_{\ell -1})$ is an independent set in $\mH''$. Thus, by the induction hypothesis, there exists $(F_2, \ldots, F_{\ell-1})$ such that  $F_j \subseteq U_j$ and $|F_j| = |V_j|p_j$ for all $2 \le j \le \ell -1$, and a suitable container $C_i = C_i(F_2, \dots, F_{\ell -1})$ for some $i \in [\ell -1]$ such that $U_i \subseteq C_i \cup F_i$ and $|C_i| \leq (1 - \zeta) |V_i|$.

\end{proof}

\section{Nearly Orthogonal Sets}
\label{section: orthogonal sets}
The following result has the role of Theorem 2.4 and Theorem 2.6 in \cite{haviv2024larger}. Unlike these two theorems, here we do not count the number of `bad' tuples of sets. Rather, we give an upper bound on the size of a family which underpins such tuples. To this end, we will need the notion of $(n,D,\lambda)$-graphs. Recall that a $D$-regular graph on $n$-vertices is an $(n,D,\lambda)$-graph if the absolute values of all but the largest eigenvalue of its adjacency matrix are at most $\lambda$.
Given a graph $G$, we denote with $\mathfrak{I}_\ell(G)$ the family of all $\ell$-tuples $(U_1, \ldots, U_\ell)$, where $U_i \subseteq V(G)$ for each $i \in [\ell]$, such that there is no copy $(v_1, \ldots, v_\ell)$ of $K_\ell$ in $G$ with $v_i \in U_i$ for $i \in [\ell]$. Note that there is no condition on disjointness of the $U_i$.

\begin{lem} \label{lemma:underpin}
    Let $c \in (0, 1/2]$ be a constant and $n$ a sufficiently large integer. Suppose $G$ is an $(n, D, \lambda)$-graph with $D \ge c n$. There exists $C = C(\ell, c) > 0$ and a family $\mC$ of subsets of $V(G)$ with the following properties:
    \begin{itemize}
        \item $|\mC| \le n^{C \log n}$;
        \item each $S \in \mC$ is of size $|S| \le C \lambda$;
        \item for every $(U_1, \ldots, U_\ell) \in \mathfrak{I}_\ell(G)$, there exists $S \in \mC$ and $i \in [\ell]$ such that $U_i \subseteq S$.
    \end{itemize}    
\end{lem}

\noindent
Note that we make no assumptions on $\lambda$ in Lemma \ref{lemma:underpin}. However, as $D \ge c n$, for $\lambda > |D|/(c \cdot C)$ we have $C \lambda > n$, thus the upper bound on $S \in \mathcal{C}$ becomes trivial.

\begin{proof}[Proof of Lemma \ref{lemma:underpin}]
Form a hypergraph $\mH$ by taking $\ell$ disjoint copies of $V(G)$, named $V_1, \ldots, V_\ell$, and putting a hyperedge on top of $(v_1, \ldots, v_\ell)$, $v_i \in V_i$ for $i \in [\ell]$, if the corresponding vertices in $G$ form a copy of $K_\ell$.

Using the Expander Mixing Lemma \cite[Corollary 9.2.5]{alon16probmethod} and a standard embedding argument\footnote{See e.g.~ \cite[Proposition 3.3]{conlon2014extremal} or \cite[Theorem 4.10]{krivelevich2006pseudo}. Note that these results show both a lower and upper bound. We only require the lower bound, which can be more easily obtained with an inductive argument.}, for any $\ell$-tuple of subsets $(W_1, \dots, W_\ell)$, $W_i \subseteq V_i$, we have 
\begin{equation} \label{eq:aux_hypergraph_edge_bound}
    e(\mH[W_1 \cup \ldots \cup W_\ell]) \geq \frac{1}{2} \prod_{i = 1}^\ell |W_i| \left( \frac{D}{n} \right)^{\binom{\ell}{2}} > c^{\ell^2} \prod_{i = 1}^\ell |W_i|,
\end{equation}
provided $|W_i| \geq C' \lambda \left( n / D \right)^{\ell -1}$ for large enough constant $C'$. 

Note that if $(U_1, \ldots, U_\ell) \in \mathfrak{I}_\ell(G)$ then $(U_1, \ldots, U_\ell)$ is an independent set in $\mH$. First of all, add to $\mC$ all subsets of $V(G)$ of size $C \log n$, where $C = C' c^{-\ell} \ell^2 \zeta^{-1}$. The meaning of the different factors in this constant will become apparent in different later steps. This takes care of all $(U_1, \ldots, U_\ell) \in \mathfrak{I}_\ell(G)$ with some $|U_i| \le C \log n$.  Consider now some $(U_1, \ldots, U_\ell) \in \mathfrak{I}_\ell(G)$ with $|U_i| > C \log n$ for all $i \in [\ell]$. Set $W_i = V_i$, $U_i' = U_i$ and $F_i = \emptyset$ for $i \in [\ell]$, and as long as each $W_i$ is of size at least $C \lambda \left( n / D \right)^{\ell-1}$, repeat the following:
\begin{itemize}
    \item Let $\mH' = \mH[W_1 \cup \ldots \cup W_\ell]$, and define a measure $\nu$ on $2^{V(\mH')}$ to be uniform on $\mH'$:

    \[ \nu(X) = \left\{ 
    \begin{array}{ll}
        \frac{1}{e(\mH')} & \text{if }  X \in \mH' \\
         0 & \text{otherwise.}
    \end{array}
    \right. \]
    By \eqref{eq:aux_hypergraph_edge_bound}, the condition \eqref{eq:spreadness condition} is satisfied for $K = c^{-\ell^2}$ and $p_i = 1 / |W_i|$ for each $i \in [\ell]$. 
    
    \item As $(U_1', \ldots, U'_\ell)$ is an independent set in $\mH'$, apply Lemma \ref{lemma:container} to obtain fingerprints $(\emptyset, F'_2 \ldots, F'_\ell)$. Let $i \in [\ell]$ be the index and $C_i \subseteq W_i$ the set corresponding to $(\emptyset, F'_2, \ldots, F'_\ell)$. Note that for each $j \in [2,\ell]$, $F'_j$ is a single vertex, and $|C_i| \le (1 - \zeta)|W_i|$, where $\zeta = \zeta(K, \ell) > 0$. Moreover, $U_i' \subseteq F'_i \cup C_i$.

    \item Set $U_i' = U_i' \cap C_i$, $W_i = C_i$, and $F_j := F_j \cup F'_j$ for $j \in [2,\ell]$.  Proceed to the next round.
\end{itemize}
As the sets $W_i$ are shrinking, the hypergraph $\mH'$ might become increasingly more unbalanced. The average degree of a vertex in $W_i$ is $\Theta(\prod_{j \neq i} |W_j|)$, meaning that if $|W_i| \ll |W_{i'}|$, then the average degree of vertices in $W_i$ is significantly larger than of those in $W_{i'}$. This dependency is captured by the parameter $p_i$ and is the main reason the standard hypergraph containers do not apply.

As one of the $W_i$'s shrinks by a factor of at least $1 - \zeta$ in each round, the whole process terminates after at most $t = \ell \zeta^{-1} \log n$ iterations. The size of each $F_i$ is thus at most $t$, and for convenience we add to each $F_i$ some elements from $U_i \setminus F_i$, in a canonical way, so that $|F_i| = t$. Let $z \in [\ell]$ denote the index of the set $W_z$ which is smaller than $C' \lambda (n / D)^{\ell - 1}$ at the end of the process. By the property \ref{container:reconstruct} of Lemma \ref{lemma:container}, we can reproduce each iteration of the process if we start with $(\emptyset, F_2, \ldots, F_\ell)$ instead of $(U_1, \ldots, U_\ell)$ (it is interesting to note that $U_1$ plays no role in the process). Therefore, $z$ and $W_z$ are indeed a function of $(F_2, \ldots, F_\ell)$. Finally, note that for each $j \in [\ell]$, $U_j \subseteq F_j \cup W_j$ (where $F_j = \emptyset$ for $j = 1$) and so specifically this is also true for $j = z$. When $n$ is large enough so that $\lambda \gg t$, we obtain 
$$
    |F_z \cup W_z| \le t + C' c^{-\ell} \lambda < C \lambda.
$$
We expand the family of subsets $\mC$ in the following way: for each $\mathbf{F} = (\emptyset, F_2, \ldots, F_\ell)$ such that $|F_2| = \ldots = |F_\ell| = t$ and $\mathbf{F}$ is the resulting fingerprint when running the process with some $(U_1, \ldots, U_\ell) \in \mathfrak{I}_{\ell}(G)$, let $i = i(\mathbf{F}) \in [\ell]$ and $W_i = W_i(\mathbf{F}) \subseteq V(G)$, and add $F_i \cup W_i$ to $\mC$. This guarantees that for every $(U_1, \ldots, U_\ell) \in \mathfrak{I}_\ell(G)$, one of the sets $U_i$ is a subset of some $S \in \mC$. The size of $\mC$ is at most
$$
    \binom{n}{C \log n} + \binom{n}{\ell \zeta^{-1} \log n}^\ell < n^{C \log n}.
$$

\end{proof}

With Lemma \ref{lemma:underpin} at hand, the proof of Theorem \ref{thm: orthogonal result} follows the structure of Haviv et al.~\cite{haviv2024larger}.

\begin{proof}[Proof of \Cref{thm: orthogonal result}.]
Let $C = C(\ell, 1/2p) > 0$ be the constant given by Lemma \ref{lemma:underpin}, and choose $t$ to be the largest integer such that $C \log(p^t) < k/6$. Furthermore, we assume $k$ is large enough such that we can apply Lemma \ref{lemma:underpin} with $n = p^t$. For small $k$ we pick $\delta(p, \ell)$ to be sufficiently small to satisfy Theorem \ref{thm: orthogonal result} when using $d$ basis vectors in $\maF_p^d$. Let $m \ge 1$  be the largest integer such that $t^m \le d$.

Let $Q$ denote the set of non-self-orthogonal distinct vectors in $\mathbb{F}_p^t$. Note that $|Q| \ge p^{t-1}$. Let $V \subseteq Q^m$ be a set of 
$$
    r = \lfloor p^{m t / (4 \ell)} \rfloor
$$ 
$m$-tuples of vectors from $Q$, chosen uniformly at random and independently. Let us denote these $m$-tuples as $(v_1^{(i)}, \ldots, v_m^{(i)})$, for $i \in [r]$. Note that it is possible that some of these $m$-tuples are identical. 

Consider the orthogonality graph $G(p,t)$ which has vertex set equal to the vectors in $\maF_p^t$ and $u \sim v$ if $\langle u,v \rangle =0$. It is well known that this graph is an $(n,D,\lambda)$-graph (e.g.~see \cite[Page 220]{alon1997constructive})  with the following parameters:
\[n = p^t-1,  \hspace{10pt}  D=p^{t-1}-1, \hspace{5pt} \text{and} \hspace{7pt} \lambda = p^{t/2-1} (p -1). \]
Since $t$ is sufficiently large, we can apply Lemma \ref{lemma:underpin}. Let $\mC$ be the family given by Lemma \ref{lemma:underpin} together with constant $C$. Recall that $|\mC| \leq n^{C \log n}$ and that all elements in $\mC$ have size at most $C \lambda$. 

Given a set $K = \{i_1, \ldots, i_k\} \subseteq [r]$ of size $|K| = k$ and $z \in [m]$, let us denote with $\mB_z(K)$ the event that $\{v_z^{(i_1)}, \ldots, v_z^{(i_k)}\} \subseteq S$ for some $S \in \mC$. By the union bound, we have
\begin{align*}
    \Pr[\mB_z(K)] 
    &\le\sum_{S \in \mC} \Pr\left[ \{v_z^{(i_1)}, \ldots, v_z^{(i_k)}\} \subseteq S \right] \\
    &= \sum_{S \in \mC} \left( \frac{|S|}{|Q|} \right)^k \le |\mC| \left( \frac{C \lambda}{p^{t-1}} \right)^k \\
    &\le 2^{C \log^2(p^t)} \left( C p^{-t/2 + 1} \right)^k < \left(C p^{-t/3 + 1} \right)^k.
\end{align*}
Let $\mB(K)$ denote the event that $\mB_z(K)$ happens for at least $m/\ell$ different values of $z \in [m]$. By union bound and the fact that the events $\mB_z(K)$ are independent for different values of $z$, we have
\begin{align*}
    \Pr[\mB(K)] \le \sum_{\substack{Z \subseteq [m] \\ |Z| \ge m/\ell}} \Pr\left[ \bigwedge_{z \in Z} \mB_z(K)\right] \le 2^m \left(C p^{-t/3 + 1} \right)^{k m / \ell} \le \left(2 C p^{-t/3 + 1} \right)^{k m / \ell},
\end{align*}
where the last inequality follows from $k \ge \ell$. Finally, let $\mB$ denote the event that $\mB(K)$ happens for some $K \subseteq [r]$ of size $k$. By one last application of the union bound, we have
$$
    \Pr[\mB] \le \sum_{\substack{K \subseteq [r] \\ |K| = k}} 
    \Pr[ \mB(K) ] \le \binom{r}{k} \left(2 C p^{-t/3 + 1} \right)^{km/\ell} \le p^{k m t  / (4\ell)} \left(2 C p^{-t/3 + 1} \right)^{km / \ell} \le \left( 2Cp^{-t/12 + 1} \right)^{km / \ell} < 1.
$$

\noindent
The last inequality holds for large enough $t$, which is implied by the assumption that $k$ is large. 

To summarise, there is a choice for $V$ for which $\mB$ does not occur. Now recall that for two vectors $u = (u_1, \ldots, u_t),v = (v_1, \ldots, v_t) \in \maF^t$, their tensor product $u \otimes v = w \in \maF^{t^2}$ is the vector whose entries are indexed by $(a,b) \in [t]^2$ and are defined by $w_{(a,b)} = u_a  v_b$. Form the set $W$ by taking the tensor product $w_i = v_1^{(i)} \otimes \ldots \otimes v_m^{(i)}$ for each $i \in [r]$. Note that all the vectors in $W$ are non-self-orthogonal, and if $\langle v_z^{(i)}, v_z^{(j)} \rangle = 0$ for some $z \in [m]$ and $i, j \in [r]$, then $\langle w_i, w_j \rangle = 0$. We claim that the corresponding set $W$ satisfies that every $\ell$-tuple of subsets $(W_1, \dots, W_\ell)$, with each $W_i$ of size $k$, contains $\ell$ vectors $\{w_i \in W_i\}_{i \in [\ell]}$  that are pairwise orthogonal. Consider one such $\ell$-tuple, and let $K_i \subseteq [r]$ denote the indices of vectors in $W_i$. As none of the events $\mB(K_i)$ happen, the set $Z_i \subseteq [m]$ of values $z \in [m]$ for which $\mB_{z}(K_i)$ happens is of size $|Z_i| < m/\ell$. Therefore, there exists $z \in [m]$ such that none of the events $\mB_z(K_1), \ldots, \mB_z(K_\ell)$ happen. Let $U_i = \{v_z^{(j)} \colon j \in K_i\}$ for each $i \in [\ell]$. Then $(U_1, \ldots, U_\ell) \not \in \mathfrak{I}_{\ell}(G(p, t))$ as otherwise we would have $U_i \subseteq S$ for some $i \in [\ell]$ and $S \in \mC$, i.e.~$B_z(K_i)$ happens for some $i$, which is not the case. This means there is $i_j \in K_j$ for each $j \in [\ell]$ such that all the vectors in $\{v_z^{(i_1)}, \ldots, v_z^{(i_\ell)}\}$ are pairwise orthogonal. By the observed property of tensor products, we consequently have that all the vectors in $\{w_{i_1}, \ldots, w_{i_\ell}\}$ are also pairwise orthogonal.

To conclude, we have obtained a set $W \subseteq \mathbb{F}_p^{mt}$ of size 
$$
    |W| = r = p^{mt/(4\ell)} = p^{\Theta(\log d \cdot t / \log t)} = d^{\Theta(k / \log k)}
$$
with the desired property.
Now we add $d - t^m$ zeroes at the end of each vector in $W$ such that we obtain a subset of $\maF_p^d$ and conclude the proof.
\end{proof}

\bibliographystyle{abbrv}
\bibliography{references}

\end{document}